\documentclass[11pt,letterpaper,reqno]{amsart}

\newcommand{\R}{\mathbb{R}}

\usepackage{amssymb}
\usepackage{amsthm}
\usepackage{amsxtra}
\usepackage{graphicx}
\usepackage{textcomp, verbatim, color, mathrsfs, braket}

\newtheorem{theorem}{Theorem}

\newtheorem{proposition}[theorem]{Proposition}
\newtheorem{lemma}[theorem]{Lemma}

\newtheorem{assumption}[theorem]{Assumption}

\theoremstyle{remark}
\newtheorem{remark}[theorem]{Remark}

\numberwithin{equation}{section}

\numberwithin{theorem}{section}

\numberwithin{table}{section}

\numberwithin{figure}{section}

\ifx\pdfoutput\undefined
  \DeclareGraphicsExtensions{.pstex, .eps}
\else
  \ifx\pdfoutput\relax
    \DeclareGraphicsExtensions{.pstex, .eps}
  \else
    \ifnum\pdfoutput>0
      \DeclareGraphicsExtensions{.pdf}
    \else
      \DeclareGraphicsExtensions{.pstex, .eps}
    \fi
  \fi
\fi

\title{On Dispersive Blow-ups for the Nonlinear Schr\"odinger Equation}
\author{Younghun Hong}
\address{The University of Texas at Austin}
\email{yhong@math.utexas.edu}

\author{Maja Taskovic}
\thanks{Mailing address: The University of Texas at Austin,
Department of Mathematics, RLM 8.100,\\
2515 Speedway Stop C1200,
Austin, Texas 78712}
\email{mtaskovic@math.utexas.edu}
\date{\today}
\begin{document}

\begin{abstract}
In this article, we provide a simple method for constructing dispersive blow-up solutions to the nonlinear Schr\"odinger equation. Our construction mainly follows the approach in Bona, Ponce, Saut and Sparber \cite{BPSS}. However, we make use of the dispersive estimate to enjoy the smoothing effect of the Schr\"odinger propagator in the integral term appearing in Duhamel's formula. In this way, not only do we simplify the argument, but we also reduce the regularity requirement to construct dispersive blow-ups. In addition, we provide more examples of dispersive blow-ups by constructing solutions that blow up on a straight line and on a sphere. 
\end{abstract}

\maketitle

\section{Introduction}

In this paper, we extend the recent work of  Bona, Ponce, Saut and Sparber \cite{BPSS} by lowering the regularity requirement and  constructing few more examples of dispersive blow-up solutions for the nonlinear Schr\"odiner equation (NLS) in $\mathbb{R}^d$ with $d\geq 2$:
\begin{equation}\label{NLS}
i \partial_t u + \Delta u \pm |u|^{p-1} u = 0,\ u(0)=u_0,
\end{equation}
where $u=u(t,x): I\times\mathbb{R}^d\to \mathbb{C}$ and $I\subset \mathbb{R}$.

The term dispersive blow-up, abbreviated by DBU, was coined by Bona and Saut \cite{BS93} to describe smooth solutions that develop a blow up at a single point. To be precise, we say that $u(t,x)$ has a \textit{dispersive blow-up} at a point $(t_*, x_*) \in \R \times \R^d$ if
\begin{equation}
\lim_{(t,x) \rightarrow (t_*, x_*)} |u(t,x)| = +\infty,
\end{equation}
and if $u(t,x)$ is continuous everywhere except at $(t_*, x_*)$. In other words, a dispersive blow up is a point-wise phenomena, while usual blow-up solutions to the dispersive PDEs typically blow-up in a certain Sobolev norm.

The theory of the dispersive singularity has its roots in the paper by Benjamin, Bona and Mahony \cite{BBM}, where it is remarked that the linearized Korteweg-de Vries equation (KdV) can develop singularity at a singe point. Bona and Saut \cite{BS93} further studied dispersive blow-ups in depth in the context of KdV. They proved the existence of solutions of KdV that develop point singularities in finite time due to the focusing effect related to the dispersive properties of the equation.

In recent years, dispersive blow-ups were extended to the framework of the Schr\"odinger equation. In \cite{BS10}, Bona and Saut analyzed dispersive blow-ups for the one-dimensional linear and nonlinear Schr\"odinger equations. Later, Bona, Ponce, Saut and Sparber \cite{BPSS} extended the theory to multiple dimensions and relaxed restriction on the nonlinearity.

The general strategy for constructing dispersive blow-ups in \cite{BBM, BS93,BS10,BPSS} consists of two steps. The first step is to construct a dispersive blow-up to the linear equation by choosing an appropriate smooth initial data explicitly. The next step is prove that the same initial data leads to dispersive blow-up for the nonlinear problem by showing that the nonlinearity does not destroy the formation of the singularity. The formation of a singularity is associated only with the dispersion relation in the linear part of the equation. This is the reason that such a singularity formation is called a ``dispersive" blow-up.

In this paper, following the general strategy in \cite{BBM, BS93,BS10,BPSS}, we construct dispersive blow-ups for NLS at an arbitrary point. Throughout the paper, we make the following assumption, which is natural since it guarantees local well-posedness of the equation in $H^s(\R^d)$, see Propostion \ref{LWP}.

\begin{assumption}\label{A}
Assume
\begin{align}
s>\frac{d}{2} - \frac{2}{p-1},\quad  s \geq 0, \quad \mbox{and} \;\; p>1.
\end{align}
When $p$ is not an even integer or when $s \geq 2$, we further assume that
\begin{equation}
\left\{
	\begin{array}{ll}
	 p>s-1, & \mbox{if} \;\; 2 \leq s < 4, \;\; \mbox{and} \;\; 1<s<\frac{d}{2}\\
	 p>s-2, & \mbox{if} \;\; s \geq 4, \;\; \mbox{and} \;\; 1<s<\frac{d}{2}\\
	 p > 1+ \lfloor s \rfloor, & \mbox{if} \;\; s \geq \frac{d}{2}.
	\end{array}
\right. 
\end{equation}
\end{assumption}

\begin{theorem}[DBU at a point]\label{MainTheorem}
Let $(t_*, x_*) \in\R\times\R^d$ with $d\geq 2$. For any $s \in ( \frac{d}{2} - \frac{2}{p}, \frac{d}{2}] $, for which $p$ and $s$ satisfy the conditions in the Assumption \ref{A}, there exists initial data $u_0 \in H^s(\R^d) \cap C^\infty(\R^d)$ such that the solution $u(t)$ to NLS with initial data $u_0$ blows up at $(t_*,x_*)$, that is,
$$\lim_{(t,x)\rightarrow (t_*, x_*)} |u(t,x)| =  + \infty.$$
Moreover, $u(t,x)$ is continuous in $\R\setminus\{t_*\} \times \R^d$, and $u(t_*, \cdot)$ is continuous on $\R^d \setminus\{x_*\}$.
\end{theorem}

\begin{remark}
The assumption $s\leq \frac{d}{2}$ is necessary since dispersive blow up does not exist in $H^s$ when $s>\frac{d}{2}$ due to the embedding $H^s \hookrightarrow L^\infty$ for $s>\frac{d}{2}$.
\end{remark}

The novelty of this theorem lies in its simpler proof. In particular, we use dispersive estimate to show that the nonlinear part does not affect the blow up of the linear part. As a consequence we lower regularity requirement on the initial data.

We also provide new examples of dispersive blow-ups by constructing solutions that blow up along a stright line in $\R^d$.

\begin{theorem}[DBU along a line]\label{MainTheorem2}
Let $t_* \in \R$ and let $l$ be a straight line in $\R^d$.  Suppose that $1<p<4$. For any $s \in ( \frac{d}{2} - \frac{2}{p}, \frac{d}{2}] $, for which $p$ and $s$ satisfy the conditions in the Assumption \ref{A}, there exists initial data $u_0 \in H^s(\R^d) \cap C^\infty(\R^d)$ such that the corresponding solution $u$ to NLS blows up along the line $l$ at time $t=t_*$, that is, for any $x_* \in l$
$$\lim\limits_{(t,x)\rightarrow (t_*, x_*)} |u(t,x)| = + \infty.$$
Moreover, $u(t,x)$ is continuous on $\R\setminus\{t_*\} \times \R^d$, and $u(t_*, \cdot)$ is continuous on $\R^d \setminus l$.
\end{theorem}

It is of interest to construct a dispersive blow-up on a compact manifold. As an example, we construct a DBU on the unit sphere in $\mathbb{R}^3$. 

\begin{theorem}[DBU on a sphere in $\R^3$]\label{MainTheorem3}
Let $t_* \in \R$. Suppose also that $1<p<2$. Then for any positive $s \in ( \frac{3}{2}-\frac{2}{p}, \frac{3}{2}] $, there exists initial data $u_0 \in H^s(\R^3) \cap C^\infty(\R^3)$ such that the solution $u$ to NLS with initial data $u_0$ blows up along the sphere $\mathcal{S}^2$ at time $t=t_*$, that is, for any $x_* \in \mathcal{S}^2$
$$\lim\limits_{(t,x)\rightarrow (t_*, x_*)} |u(t,x)| = + \infty.$$
Moreover, $u(t,x)$ is continuous on $\R\setminus\{t_*\} \times \R^3$, and $u(t_*, \cdot)$ is continuous on $\R^3 \setminus\mathcal{S}^2$.
\end{theorem}

The key ingredient for proving Theorem \ref{MainTheorem}, \ref{MainTheorem2} and \ref{MainTheorem3} is the smoothing estimate in Proposition \ref{smoothing}. It will be used in the second step of the general strategy, that is, for controlling the nonlinear term. Proposition \ref{smoothing} simplifies and improves the corresponding proposition in Bona-Ponce-Saut-Sparber \cite{BPSS} (Proposition 4.1) in lowering the regularity requirement from $s \in ( \frac{d}{2} - \frac{1}{2(p-1)}, \frac{d}{2}]$ to $s \in ( \frac{d}{2} - \frac{2}{p}, \frac{d}{2}]$. This is achieved via the use of  the dispersive estimate.

We remark that reducing regularity requirement in the smoothing estimate is not just a purely mathematical issue. It is helpful for including higher power nonlinearities in construction of a dispersive blow-up along a line. Indeed, by the argument in Section 3 with the previously known smoothing estimate in \cite{BPSS}, one can construct a nonlinear dispersive blow-up along a line only for sub-quadratic nonlinearity, i.e., $1<p<2$. It is due to the fact that in the construction, we use an initial data making a dispersive blow-up in $\mathbb{R}^{d-1}$ (see \eqref{varphi}). However, by the smoothing estimate in Proposition \ref{smoothing}, one can include any  nonlinearity $1<p<4$. Similarly, the condition on $p$ in Theorem \ref{MainTheorem3} is also relaxed by a smoothing estimate with a low regularity.

\subsection*{Outline of the paper} We prove Theorem \ref{MainTheorem} in Section 2, Theorem \ref{MainTheorem2} in Section 3. and Theorem \ref{MainTheorem3} in Section 4.

\section{Construction of a DBU at a Point: Proof of Theorem \ref{MainTheorem}}

In this section, we prove the main theorem (Theorem \ref{MainTheorem}) following the approach in Bona, Ponce, Saut and Sparber \cite{BPSS}. Before we begin, we note that both linear and nonlinear Schr\"odinger equations are invariant under the spatial translation, i.e., if $u(t,x)$ solves the equation, then $u(t,x-x_*)$ solves the same equation for any $x_*\in\mathbb{R}^d$. Thus, it suffices to construct a DBU that blows up at the origin in $\mathbb{R}^d$.

The proof of Theorem \ref{MainTheorem} proceeds in two steps: (1) Construction of a DBU for the linear equation. (2) Proving that the nonlinearity does not break singularity formation. 
\subsection{Construction of a linear DBU} Consider the linear Schr\"odinger equation
\begin{equation}\label{LS}
i \partial_t u  + \Delta u = 0,\ u(0) = u_0.
\end{equation}
By the following proposition, one can explicitly construct a DBU to the linear equation \eqref{LS} at the origin at an arbitrary time. 
\begin{proposition}[Linear DBU at a point; \cite{BPSS}, Lemma 2.1]\label{linear DBU at a point}
Let $u(t)$ be the solution to the linear Schr\"odinger equation with initial data 
\begin{equation}\label{initial data for DBU at a point}
u_0(x)=\frac{\alpha e^{-i|x|^2/4t_*}}{(1+|x|^2)^m}\in C^\infty(\mathbb{R}^d)\cap L^2(\mathbb{R}^d),
\end{equation}
where $\alpha\in\mathbb{R}$, $t_*\in\mathbb{R}$ and $m\in(\frac{d}{4},\frac{d}{2}]$. Then, $u(t,x)$ blows up at $(t_*,0)$, i.e., $|u(t,x)| \to + \infty$ as $(t,x)\to (t_*,0)$. Moreover, $u(t,x)$ is continuous in $ \R \setminus \{t_*\} \times \R^d$, and $u(t_*, x)$ is continuous in $\R^d \setminus{\{0\}}$.
\end{proposition}

\begin{proof}[Sketch of the Proof]
\noindent The well-known representation of the solution to the linear problem \eqref{LS}
\begin{align}\label{lin-sol}
u(t,x) \; = \; \frac{1}{(4 \pi i t)^{d/2}} \; \int_{\R^d} e^{i \frac{|x-y|^2}{4t}} \; u_0(y) \; dy,
\end{align}
implies that
\begin{align*}
u(t_*, x) \, 
& = \, \frac{1}{(4\pi i t_*)^{d/2}} \, e^{\frac{i|x|^2}{4t_*}} 
\int_{\R^d} e^{\frac{-2i x \cdot y}{4t_*}} \, \frac{\alpha}{(1+|y|^2)^m} \; dy.
\end{align*}
The above integral is the Fourier transform of a Bessel potential (see e.g. \cite{Stein}) . Therefore
\begin{align} \label{K}
u(t_*, x) \, = \,  \frac{C\, \alpha}{(4\pi i t_*)^{d/2}}  \;  e^{\frac{i|x|^2}{4t_*}} \; |\tfrac{x}{4t_*}|^{-\nu} \, K_\nu (|\tfrac{x}{4t_*}|),
\end{align}
where $K_\nu(|x|)$ is the Fourier transform of $\frac{1}{(1+|y|^2)^m}$ and  $\nu = \frac{d}{2} - m$. By the properties of the Bessel potential, $K_\nu$ behaves as $K_\nu(|x|) \sim \frac{C}{|x|^\nu}$ for $\nu >0$ and $K_0(|x|) \sim -\log(|x|)$ as $x \rightarrow 0$. Therefore
\begin{align*}
& u(t_*, x) \sim |x|^{-(d-2m)} , \quad \mbox{as} \; x \rightarrow 0, \quad \mbox{if}\;  d-2m>0, \\
& u(t_*, x) \sim -\log(|x|)  \quad \mbox{as} \; x \rightarrow 0, \quad \mbox{ if} \; d-2m=0.
\end{align*}
Since $2m \leq d$, we conclude that $u(t_*, 0) = \infty$. Also, by \eqref{K} and the properties of the Bessel potential, the continuity of $u(t_*, x)$ in $\R^d \setminus{\{0\}}$ follows. The continuity of  $u(t,x)$ in $ \R \setminus \{t_*\} \times \R^d$ can be proved as in [\cite{BS10}, Theorem 2.1.].
\end{proof}

\subsection{Construction of a nonlinear DBU}
Now, we consider the Cauchy problem for NLS \eqref{NLS}. It is well-known that this equation is locally well-posed in a sub-critical space (see \cite{Caz} for instance, and \cite{P} for relaxed conditions in case of non-integer $p$). To be precise, we define the solution to the NLS with initial data $u_0$ as the solution to the integral equation
\begin{equation}\label{Duhamel}
u(t)=e^{it\Delta}u_0\pm i\int_0^t e^{i(t-s)\Delta}(|u|^{p-1}u)(s)ds.
\end{equation}
We call $(q,r)$ \textit{admissible} if $2\leq q,r\leq\infty$, $(q,r,d)\neq(2,\infty, 2)$ and 
$$\frac{2}{q}+\frac{d}{r}=\frac{d}{2}.$$

\begin{proposition}[Local well-posedness in $H^s(\R^d)$, \cite{Caz, P}]\label{LWP}
Let $p$ and $s$ satisfy the conditions in Assumption \ref{A}.  For such choice of $s$, we have: \\
$(i)$ If $u_0\in H^s$, then there exist a time interval $I=I(u_0)$ and a unique solution $u(t)\in C_{t\in I}H_x^s\cap L_{t\in I}^qW_x^{s,r}$ to the NLS with initial data $u_0$ for all admissible pairs $(q,r)$.\\
$(ii)$ If $\|u_0\|_{H^s}$ is sufficiently small, then $u(t)$ exists in $H^s$ globally in time.
\end{proposition}

Following the approach in Bona-Ponce-Saut-Sparber \cite{BPSS}, we aim to show that the initial data $u_0$, given by \eqref{initial data for DBU at a point}, produces a DBU for the nonlinear problem at the same point and at the same time as for the linear problem. To this end, first, we need to make sure that $u_0\in H^s$ for $s>\frac{d}{2}-\frac{2}{p-1}$, because otherwise, we do not know that the solution to NLS exists. The following lemma provides a sufficient condition.

\begin{lemma}[\cite{BPSS}, Lemma 3.3.]\label{nonlinear initial data}
The initial data $u_0$, given in \eqref{initial data for DBU at a point}, is contained in $H^s$ if $2m > s + \frac{d}{2}$.
\end{lemma}

Next, we show that the integral term in the Duhamel formula \eqref{Duhamel} stays bounded during the existence time.

\begin{proposition}[Smoothing estimate]\label{smoothing} Let $s>\frac{d}{2} - \frac{2}{p}$. Suppose that $u(t)$ solves NLS \eqref{NLS} on a time interval $I$. Then, we have
\begin{equation}\label{eq: smoothing}
\Big\|\int_0^t e^{i(t-s)\Delta}(|u|^{p-1}u)(s)ds\Big\|_{L_{t\in I}^\infty C^0_x}\leq |I|^{\frac{4+p(2s-d)}{6}}\|u\|_{L_{t\in I}^\infty H_x^s}^p.
\end{equation}
\end{proposition}

The proof of the smoothing estimate is based on the dispersive estimate.
\begin{lemma}[Dispersive estimate]\label{dispersive} For $2\leq r\leq \infty$, we have
$$\|e^{it\Delta}f\|_{L^r(\R^d)}\lesssim |t|^{-d(\frac{1}{2}-\frac{1}{r})}\|f\|_{L^{r'}(\R^d)},$$ 
where $r'$ is a H\"older conjugate of $r$, that is $\frac{1}{r} + \frac{1}{r'} =1$.
\vspace{5pt}
\end{lemma}

\begin{proof}[Proof of Proposition \ref{smoothing}]
Let $\epsilon>0$ be an arbitrarily small number such that $\frac{dp  - 4+3\epsilon}{2p}\leq s$. Then, applying the Sobolev inequality, Minkowski inequality and the dispersive estimate, respectively, we obtain
\begin{align*}
\Big\|\int_0^t e^{i(t-s)\Delta}F(s)ds\Big\|_{L_{t\in I}^\infty C^0_x}&\lesssim \Big\|\int_0^t e^{i(t-s)\Delta}F(s)ds\Big\|_{L_{t\in I}^\infty W_x^{\frac{d-2}{2}+\epsilon,\frac{2d}{d-2+\epsilon}}}\\
&\leq \Big\|\int_0^t \|e^{i(t-s)\Delta}F(s)\|_{W_x^{\frac{d-2}{2}+\epsilon,\frac{2d}{d-2+\epsilon}}}ds\Big\|_{L_{t\in I}^\infty }\\
&\lesssim \Big\|\int_0^t \frac{1}{|t-s|^{\frac{2-\epsilon}{2}}}\|F(s)\|_{W_x^{\frac{d-2}{2}+\epsilon,\frac{2d}{d+2-\epsilon}}}ds\Big\|_{L_{t\in I}^\infty}.
\end{align*}
Note that the parameters in the above Sobolev inequality were chosen so that the power of $|t-s|$ obtained after application of the dispersive estimates is just barely integrable. We next integrate  the power of $|t-s|$, and obtain
\begin{equation}\label{smoothing estimate}
\begin{aligned}
\Big\|\int_0^t e^{i(t-s)\Delta}F(s)ds\Big\|_{L_{t\in I}^\infty C^0_x}&\lesssim \Big(\int_0^t \frac{1}{|t-s|^{\frac{2-\epsilon}{2}}}ds\Big)\|F\|_{L_{t\in I}^\infty W_x^{\frac{d-2}{2}+\epsilon,\frac{2d}{d+2-\epsilon}}}\\
&\lesssim |I|^{\frac{\epsilon}{2}}\|F\|_{L_{t\in I}^\infty W_x^{\frac{d-2}{2}+\epsilon,\frac{2d}{d+2-\epsilon}}}.
\end{aligned}
\end{equation}
Now, by applying  the fractional Leibnitz rule to \eqref{smoothing estimate} we obtain
$$\Big\|\int_0^t e^{i(t-s)\Delta}(|u|^{p-1}u)(s)ds\Big\|_{L_{t\in I}^\infty C^0_x}\lesssim |I|^{\frac{\epsilon}{2}}\|u\|_{L_{t\in I}^\infty W^{\frac{d-2}{2}+\epsilon, \frac{2dp}{p(d-2 +2\epsilon) - 3\epsilon +4}}}\|u\|_{L_{t\in I}^\infty L_x^\frac{2dp}{(p-1)(4-3\epsilon)}}^{p-1}.$$
Finally, another application of the Sobolev inequality leads to
$$\Big\|\int_0^t e^{i(t-s)\Delta}(|u|^{p-1}u)(s)ds\Big\|_{L_{t\in I}^\infty C^0_x}\lesssim |I|^{\frac{\epsilon}{2}}\|u\|_{L_{t\in I}^\infty H_x^{\frac{dp-4+3\epsilon}{2p}}}^p.$$
\end{proof}

\begin{remark}
$(i)$ The proof of Proposition \ref{smoothing} is very simple, but it improves the corresponding proposition in Bona, Ponce, Saut and Sparber \cite{BPSS} (see Proposition 4.1) in lowering the regularity requirement. The key in the proof is a Kato smoothing type estimate \eqref{smoothing estimate}. This estimate is related to the fact that the inhomogeneous term 
$$\int_0^t e^{i(t-s)\Delta}F(s)ds$$
satisfies better Strichartz estimates (see \cite{Fos}, for example). In \eqref{smoothing estimate}, we exploit the smoothing property of the linear Schr\"odinger flow by the dispersive estimate as long as the bound is integrable in time.\\
$(ii)$ The smoothing estimate \eqref{eq: smoothing} is sharp (except the endpoint case $s=\frac{d}{2}-\frac{2}{p}$) in the sense that if $s<\frac{d}{2}-\frac{2}{p}$, the inequality \eqref{eq: smoothing} is super-critical.
\end{remark}

\begin{proof}[Proof of Theorem \ref{MainTheorem}]
Fix any $t_*\in\mathbb{R}$. Combining conditions on $s, p,d$ and $m$ in Proposition \ref{linear DBU at a point}, \ref{LWP} and \ref{smoothing} and Lemma \ref{nonlinear initial data}, we have
$$\max\Big\{\frac{d}{2}-\frac{1}{p},\frac{d}{4}\Big\}<m\leq\frac{d}{2}.$$
It is obvious that one can always find $m$ that satisfies these bounds. Pick any such $m$ and let $u_0$ be initial data defined by \eqref{initial data for DBU at a point}. Let $u(t)$ be the solution to NLS with initial data $u_0$. Here, by choosing sufficiently small $\alpha$ (see \eqref{initial data for DBU at a point}), one may assume that $t_*$ is contained in the existence time of $u(t)$. Then, as a consequence of Proposition \ref{linear DBU at a point} and \ref{smoothing}, the linear part in the integral equation \eqref{Duhamel} blows up at $(t_*,0)$ and the nonlinear part remains bounded during the existence time. Therefore, we conclude that $u(t,x)$ blows up at $(t_*,0)$.
\end{proof}

\section{Construction of a DBU along a Line: Proof of Theorem \ref{MainTheorem2}}

In this section, we provide a new example of a dispersive blow-up, namely along a line.  Since both linear and nonlinear Schr\"odinger equations are invariant under the spatial translation and rotation, it suffices to construct a DBU along the line
\begin{equation} \label{line}
l_1 : = \{(x_1, 0, ..., 0) \in \R^d \; : \; x_1 \in \R\}.
\end{equation}
Following the general strategy, the DBU is first constructed for the linear problem. The idea is to find initial data that leads to a solution of the linear Schr\:odinger equation which  has a dispersive blow up in $d-1$ variables at  $(x_2, x_3, ..., x_d) = (0, 0, ..., 0) $ and which is regular in the first variable $x_1$. This will yield a DBU for the linear problem along the straight line $\{(x_1, 0, ..., 0) \in \R^d \; : \; x_1 \in \R\}$. Finally, the smoothing estimate Proposition \ref{smoothing} is used to  control the nonlinear part in the NLS and show that the same initial data generates a DBU solution for the nonlinear problem as well.

Throughout this section, we use the notation
\begin{align*}
x=(x_1, x_2, ..., x_d) = (x_1, \bar{x})\in\mathbb{R}\times\mathbb{R}^{d-1},\quad\bar{x} = (x_2, ..., x_d)\in\mathbb{R}^{d-1}.
\end{align*}

\subsection{Construction of a linear DBU} As a first step, we construct initial data that leads to a DBU solution to the linear problem \eqref{LS} that blows up along the line  $l_1$ \eqref{line}. The idea is to use initial data $u_0(x)=u_0(x_1,\bar{x})$ that separates variables $x_1$ and $\bar{x}$ in such a way that the factor depending on $x_1$ is regular, while the factor depending on $\bar{x}$ develops DBU at the point $\bar{0} \in \R^{d-1}$. Such an initial data will lead to DBU along the line $l_1$. 
\begin{proposition}[Linear DBU along a straight line] \label{linear DBU along a line}
 For any $y_1\in\mathbb{R}$, any $\bar{y} \in \R^{d-1}$ and $m \in \left( \frac{d-1}{4}, \frac{d-1}{2}\right]$, define the regular and the blow-up factor respectively by
$$\varphi_\textup{reg}(z)=e^{-z^2}, \qquad \varphi_\textup{DBU}(\bar{y})=\Big(\frac{1}{(1+\cdot^2)^m}\Big)^\wedge (\bar{y}).$$
For  some $\alpha \in \R$ and any $y \in \R^d$ let 
\begin{equation}\label{varphi}
\varphi_0(y)= \alpha\;\varphi_\textup{reg}(y_1) \; \varphi_\textup{DBU}(\bar{y}).
\end{equation}
Define the initial data $u_0$ by
\begin{equation}\label{i.d. for a line}
u_0(x) \;=\;  e^{-\frac{i|x|^2}{4t_*}} \; \varphi_0^\vee(x) \;=\; \left(e^{-z^2}\right)^\vee(x_1)  \; \; \frac{ \alpha \;  e^{-\frac{i|x|^2}{4t_*}}}{(1 + |\bar{x}|^2)^m}.
\end{equation}
Let $u(t)$ be the solution to the linear Schr\"odinger equation with the initial data $u_0$ \eqref{i.d. for a line}. Then, $u(t,x)$ blows up along the line  $l_1$ \eqref{line}, that is, for any $x_* \in l_1$, $|u(t,x)| \to + \infty$ as $(t,x)\to (t_*, x_*)$. Moreover, $u(t,x)$ is continuous in $ \R \setminus \{t_*\} \times \R^d$, and $u(t_*, x)$ is continuous in $\R^d \setminus l_1$.
\end{proposition}

\begin{proof}
The function $\varphi_\textup{reg}$ is smooth, while the function $\varphi_\textup{DBU}(\bar{y})$ is a Fourier transform of a Bessel potential. To ensure that  $\varphi_\textup{DBU}(\bar{y})$ has singularity only at $\bar{0}$, we need $m \leq \tfrac{d-1}{2}$ similarly to the discussion in Section 2. We also need that $m>\tfrac{d-1}{4}$ so that  $\varphi_\textup{DBU}^\vee$ is square integrable. Thus the bounds on $m$ in the statement of the proposition.

Using the representation  \eqref{lin-sol} for the solution $u(t,x)$, we get
\begin{align*}
u\Big(t_*, x\Big)
&=\frac{1}{(4\pi it_*)^{\frac{d}{2}}} \; e^{\frac{i|x|^2}{4t_*}} \; \int_{\mathbb{R}^d} e^{-iy\cdot \frac{x}{2t_*}} \varphi_0^\vee(y)dy\\
&=\frac{ e^{\frac{i|x|^2}{4t_*}}}{(4\pi it_*)^{\frac{d}{2}}} \; \; \varphi_0\Big(\frac{x}{2t_*}\Big)\\
&=\frac{\alpha \;  e^{\frac{i|x|^2}{4t_*}}}{(4\pi it_*)^{\frac{d}{2}}} \; \varphi_\textup{reg}\Big(\frac{x_1}{2t_*}\Big) \;\varphi_\textup{DBU}\Big(\frac{\bar{x}}{2t_*}\Big).
\end{align*}
Since $\varphi_\textup{reg}(\tfrac{x_1}{2t_*})$ is bounded for all $x_1 \in \R$ and $\varphi_\textup{DBU}(\bar{x})$ blows up as $\bar{x}\to\bar{0}$, we conclude that for any $x_1\in\mathbb{R}$, 
$$\big|u(t_*, x_1,\bar{x})\big|\to +\infty\quad\textup{ as }\bar{x}\to \bar{0}.$$
Thus, initial data $u_0$ leads to a solution with dispersive blow up along the straight line $l_1$. 
 The continuity of $u(t_*, x)$ in $\R^d \setminus{\{l_1\}}$ follows by properties of Bessel potential $\varphi_\textup{DBU}$ and smoothness of $\varphi_\textup{reg}$. The continuity of  $u(t,x)$ in $ \R \setminus \{t_*\} \times \R^d$ can be proved as in [\cite{BS10}, Theorem 2.1.].
\end{proof}

\subsection{Construction of a nonlinear DBU} We now show that the the initial data $u_0$ \eqref{i.d. for a line} will lead to the dispersive blow ups solution for the  NLS as well, along the same line and at the same time. 
\begin{proof}[Proof of Theorem \ref{MainTheorem2}]
 Fix any $t_*\in\mathbb{R}$. Due to the dimension reduction in the DBU part of the initial data \eqref{i.d. for a line}, Lemma \ref{nonlinear initial data} implies that $u_0\in H^s(\R^d)$ if $2m > s + \tfrac{d-1}{2}$. Combining this with $m \in (\tfrac{d-1}{4}, \tfrac{d-1}{2}]$ with conditions on $s, p,d$ and $m$  in Proposition \ref{LWP} and Proposition \ref{smoothing} we get
$$\max\Big\{\frac{d}{2}-\frac{1}{p} - \frac{1}{4}, \;\frac{d-1}{4}\Big\}<m\leq\frac{d-1}{2}.$$
$m$ that satisfies these bounds always exists. Pick one such $m$ and let $u_0$ be the initial data given by the formula \eqref{i.d. for a  line}. By choosing sufficiently small $\alpha$ (see \eqref{varphi}), one may assume that $t_*$ is contained in the  interval of existence for the solution to NLS with initial data $u_0$. Then, as a consequence of Proposition \ref{linear DBU along a line} and \ref{smoothing}, the linear part in the integral equation \eqref{Duhamel} blows up at $(t_*,0)$ and the nonlinear part remains bounded during the existence time. Therefore, we conclude that $u(t,x)$ blows up at $(t_*,0)$.
\end{proof}

\section{Construction of a DBU on a Sphere in $\mathbb{R}^3$}

In this section, we provide one more example of a dispersive blow up, by constructing a solution to the NLS in three dimensions that blows up on the unit sphere. We follow the general strategy and first consider the linear problem. The idea for constructing a linear DBU on a sphere is very simple. Under the radially symmetric assumption, one can easily transform the 3-dimensional linear Schr\"odinger equation into the 1-dimensional equation. Then, by Proposition \ref{linear DBU at a point}, we construct a DBU at a point on $\mathbb{R}$, say $r=1$. Translating it back to the 3-dimensional equation, we obtain a DBU on the unit sphere for the linear problem. Finally,  the nonlinear part is controlled via the smoothing estimate Proposition \ref{smoothing}.

\subsection{Construction of a linear DBU} As a first step, we construct a linear DBU that blows up along the unit sphere in $\R^3$.  Here is a precise statement of the claim.
\begin{lemma}[Linear DBU on the unit sphere $\mathcal{S}^2$ in $\R^3$]\label{linear DBU on a sphere}
Fix an arbitrary $t_* \in \R$. There exists initial data $u_0 \in C^\infty(\R^3) \cap L^2(\R^3)$  such that the corresponding global solution $u(t,x)$ to the Cauchy problem \eqref{LS} blows up along the sphere $\mathcal{S}^2$, i.e., for any $x_* \in \mathcal{S}^2$, $|u(t,x)| \to + \infty$ as $(t,x)\to (t_*,x_*)$. Moreover, $u(t,x)$ is continuous in $ \R \setminus \{t_*\} \times \R^3$, and $u(t_*, x)$ is continuous in $\R^3 \setminus \mathcal{S}^2$.
\end{lemma}
\begin{proof} We will use radial symmetry  to construct such DBU. Suppose that $u$ is a radially symmetric solution to the linear Schr\"dingier equation:
$$i\partial_t u(t,x)+\Delta u(t,x)=i\partial_t u(t,r)+\partial_r^2 u(t,r)+\frac{2}{r}\partial_ru(t,r)=0.$$
Let $v(t,r)$ be the even extension of $ru(t,r)$. Then, $v(t)$ solves the $1$-dimensional linear Schr\"odinger equation 
\begin{equation}\label{1dSchrodinger}
i\partial_t v+\partial_r^2v=r\Big(i\partial_t u+\partial_r^2 u+\frac{2}{r}\partial_r u\Big)=0.
\end{equation}
Let $v(t,r)$ be the solution to \eqref{1dSchrodinger} with the initial data
\begin{align}\label{v0}
v_0(r)= \frac{\alpha \; e^{-\frac{i}{4t_*}|r-1|^2}}{(1 + r^2)^m} \, + \, \frac{\alpha \;e^{-\frac{i}{4t_*}|r+1|^2}}{(1 + r^2)^m}, \quad m\in\left(\frac{1}{4}, \frac{1}{2}\right].
\end{align}
By Proposition \ref{linear DBU at a point}, this initial data in one dimension  forms a dispersive blow up at $(t_*,1)$ (and leads to a solution with the desired continuity properties). Thus, we conclude that for all $|x|=1$,
$$u(t_*,x) = v\Big(t_*,1\Big)=\infty.$$
\end{proof}

\subsection{Construction of a nonlinear DBU}
We now show that radial initial data
\begin{align}\label{u0 3}
u_0(x) = \frac{v_0( |x|)}{|x|},
\end{align}
where $v_0$ is given by \eqref{v0} leads to the solution to the NLS with a DBU on the unit sphere. The proof again relies on the smoothing estimate Proposition \ref{smoothing} for controlling the nonlinear term, but first we need to make sure parameters $s, p$ and $m$ are chosen so that $u_0 \in H^s(\R^3)$ and that the dispersive blow up develops. 

The following Lemma, which is proven in the Appendix, is useful in finding a sufficient condition for $u_0 \in H^s(\R^3)$.

\begin{lemma}\label{app-sphere}
Let  $u_0(x) = \frac{v_0( |x|)}{|x|}$, where $v_0$  is given \eqref{v0}. If $v_0 \in H^s(\R)$ with $s \in [0, \tfrac{1}{2})$, then $u_0 \in H^s(\R^3)$.
\end{lemma}

 Now, using  Lemma \ref{nonlinear initial data} to make sure that $v_0 \in H^s(\R)$, which in turn ensures that $u_0 \in H^s(\R^3)$, and combining that with  Proposition \ref{LWP}\, and Proposition \ref{smoothing}\, we get
\begin{align}\label{m3}
\max\Big\{1-\frac{1}{p}, \;\frac{1}{4}\Big\}<m\leq\frac{1}{2}.
\end{align}
For such $m$ to exists, $p$ needs to satisfy $p<2$. Hence the condition on the order of nonlinearity in the statement of the Theorem \ref{MainTheorem3}. We are now ready to prove Theorem \ref{MainTheorem3}.
\begin{proof}[Proof of Theorem \ref{MainTheorem3}]
 Suppose $1<p<2$ and let $m$ be such that \eqref{m3} holds. For such an $m$ define initial data $u_0$ as in \eqref{u0 3} and let $u(t,x)$ to be the corresponding solution to the NLS. Again, by choosing sufficiently small $\alpha$ (see \eqref{v0}), one may assume that $t_*$ is contained in the  interval of existence of $u$. Then, as a consequence of Lemma \ref{linear DBU on a sphere} and Proposition \ref{smoothing}, the linear part in the integral equation \eqref{Duhamel} blows up at $(t_*,x)$ for all $|x|=1$, and the nonlinear part remains bounded during the existence time. Therefore, we conclude that $u(t,x)$ blows up at the unit sphere at time $t_*$.
\end{proof}

\begin{remark}
We restrict our attention to three dimensions, because if $u(t,x)$ is a radial solution to the linear Schr\"odinger equation in $\mathbb{R}^d$, then $v(t,r)  = r^{\frac{d-1}{2}} u(t,r)$ solves the 1-dimensional Schr\"odinger equation with an extra term unless $d=3$.
\end{remark}

\appendix

\section{Proof of Technical Lemma}

In the appendix, we prove a technical  Lemma \ref{app-sphere}, that provides sufficient condition for the initial data used in the construction of a dispersive blow up on a sphere to be in $H^s$.

\subsection{Initial data for DBU on a sphere } Here we prove Lemma \ref{app-sphere} which provides a sufficient condition for the radial initial data $u_0(x) = \frac{v_0( |x|)}{|x|}$, where $v_0$  is given \eqref{v0}, to be in $H^s(\R^3)$.
\begin{proof}[Proof of Lemma \ref{app-sphere}]
By changing to spherical coordinates, it is easy to check that
$$\|u_0\|_{L^2(\mathbb{R}^3)}=\Big\|\frac{v_0(|x|)}{|x|}\Big\|_{L^2(\mathbb{R}^3)}\sim\|v_0\|_{L^2(\mathbb{R})}.$$
Moreover, by the fractional Leibniz rule \cite[Proposition 1]{NP},
\begin{align*}
\|u_0\|_{\dot{H}^s(\mathbb{R}^3)}&=\Big\||\nabla|^s\Big(\frac{v_0(|x|)}{|x|}\Big)\Big\|_{L^2(\mathbb{R}^3)}\leq \Big\|\frac{(|\nabla|^sv_0)(|x|)}{|x|}\Big\|_{L^2(\mathbb{R}^3)}+\Big\|\Big(|\nabla|^s\frac{1}{|x|}\Big)v_0(|x|)\Big\|_{L^2(\mathbb{R}^3)}\\
&\lesssim\||\nabla|^sv_0\|_{L^2(\mathbb{R})}+\Big\|\frac{v_0(|x|)}{|x|^{1+s}}\Big\|_{L^2(\mathbb{R}^3)}\lesssim\||\nabla|^sv_0\|_{L^2(\mathbb{R})},
\end{align*}
where in the last step, we used H\"older inequality and Sobolev inequality in Lorentz norms (see \cite{Gra}, for example) to get
$$\Big\|\frac{v_0(x)}{|x|^{1+s}}\Big\|_{L^2(\mathbb{R}^3)}\sim \Big\|\frac{v_0(r)}{|r|^s}\Big\|_{L^2(\mathbb{R})}\lesssim \|v_0\|_{L^{\frac{2}{1-2s},2}(\mathbb{R})}\lesssim\||\nabla|^sv_0\|_{L^2(\mathbb{R})} .$$
\end{proof}

\end{document}